\documentclass{amsart}

\usepackage[inactive]{srcltx} 
\usepackage{amsmath, amsthm, amscd, amsfonts, amssymb, graphicx, color}
 \newtheorem{thm}{Theorem}[section]
 \newtheorem{cor}[thm]{Corollary}
 \newtheorem{lem}[thm]{Lemma}
 \newtheorem{prop}[thm]{Proposition}
 \theoremstyle{definition}
 
 \theoremstyle{remark}
 
\newcommand{\M}{\mathbb{M}}
\newcommand{\A}{\mathbb{A}}

\newcommand{\X}{\mathbb{X}}
\newcommand{\Hh}{\mathbb{H}}
\newcommand{\p}{P_{1}}
\newcommand{\q}{P_{2}}

\begin{document}

\title[continuous linear maps on reflexive algebras behaving likes ...]
 {continuous linear maps on reflexive algebras behaving like Jordan left derivations at idempotent-product
elements}

\author{ B. Fadaee \quad  and \quad H. Ghahramani}
\thanks{{\scriptsize
\hskip -0.4 true cm \emph{MSC(2010)}:  47B47; 47L35; 47B49.
}}

\address{Department of
Mathematics, University of Kurdistan, P. O. Box 416, Sanandaj,
Iran.}

\email{h.ghahramani@uok.ac.ir; hoger.ghahramani@yahoo.com}

\address{Department of
Mathematics, University of Kurdistan, P. O. Box 416, Sanandaj,
Iran.}


\address{}

\email{}

\thanks{}

\thanks{}



\date{}

\dedicatory{}

\commby{}


\begin{abstract}
Let $\A$ be a Banach algebra with unity $\textbf{1}$ and $ \M $ be a unital Banach left $ \A $-module. let $ \delta: \A \rightarrow \M$ be a continuous
linear map with the property that
\[ a,b\in \A, \quad ab+ba=z \Rightarrow 2a\delta(b)+2b\delta(a)=\delta(z),  \]
where $z\in \A$. In this article, first we characterize $\delta$ for $z=\textbf{1}$. Then we consider the case $\A=\M=Alg \mathcal{L}$, where $Alg \mathcal{L}$ is areflexive algebra on a Hilbert space $ \Hh $ and $z=P$ is a non-triavial idempotent in $\A$ with $P(\Hh) \in \mathcal{L}$ and describe $\delta$. Finally we apply the main results to $CSL$-algebras, irreducible $CDC$ algebras and nest algebras on a Hilbert space $\Hh$.
\end{abstract}

\maketitle

\section{Introduction}
Throughout this paper all algebras and vector spaces will be over
$\mathbb{F}$, where $\mathbb{F}$ is either the real field
$\mathbb{R}$ or the complex field $\mathbb{C}$. Let $\A$ be an
algebra with unity $1$, $\M$ be a left $\A$-module and $\delta:
\A\rightarrow \M$ be a linear mapping. $\delta$ is said to be a \emph{Jordan left derivation} if
$\delta(ab+ba)=2a\delta(b)+2b\delta(a)$ for any $a,b \in \A$, or equivalently,
if $\delta(a^{2})=2a\delta(a)$ for any $a \in \A$.
The concepts of Jordan left derivation was introduced by
Bre$\check{\textrm{s}}$ar and Vukman in \cite{Bre}. For results
concerning Jordan left derivations we refer
the readers to \cite{Gho} and the references therein.
\par
In recent years, several authors studied the linear (additive)
maps that behave like homomorphisms, derivations or left
derivations when acting on special products (for instance, see
\cite{Bre2, Gha0, Hou, Ji} and the references therein). In
this article we study the continuous linear maps on reflexive algebras behaving like Jordan left derivations at idempotent-product. In fact we consider the following condition on a continuous linear map $\delta$ from a Banach algebra $\A$ into a  Banach left $\A$-module $\M$:
\[ a,b\in \A, \quad ab+ba=z \Rightarrow 2a\delta(b)+2b\delta(a)=\delta(z),  \]
where $z\in \A$. First we describe $\delta$ on a Bnach algebra $\A$ with unity $\textbf{1}$, when $z=\textbf{1}$. Then we assume that $\A=\M=Alg \mathcal{L}$, where $Alg \mathcal{L}$ is areflexive algebra on a Hilbert space $ \Hh $ and $z=P$ is a non-triavial idempotent in $\A$ with range
$P(\Hh) \in \mathcal{L}$ and characterize $\delta$. Finally we apply the main results to $CSL$-algebras, irreducible $CDC$ algebras and nest algebras on a Hilbert space $\Hh$.
\par
The following are the notations and terminologies which are used
throughout this article.
\par
Let $\A$ be a Banach algebra with unity $\textbf{1}$. Denote by $Inv(\A)$ the set of invertible elements of $\A$. $Inv(\A)$ is an open subset of $\A$ and hence it is a disjoint union of open connected subsets, the components of $Inv(\A)$. The component containing $\textbf{1}$ is called the \emph{principal component} of $Inv(\A)$ and it is denoted by $Inv_{0}(\A)$. We denote by $e^{\A}$ the range of the exponential function in $\A$, i.e.
\[ e^{\A}=\{ e^{a}\, | \, a\in \A\}\]
and we have $e^{\A}\subseteq Inv_{0}(\A)$.
\par
Let $\Hh$ be a Hilbert space. We denote by $\mathcal{B}(\Hh)$ the
algebra of all bounded linear operators on $\Hh$. Also the identity operator on $\Hh$ is denoted by $I$ and the projection of $ \Hh$ onto the closed subspace $\mathcal{L}$ is denoted by $P_{L}$. A \emph{subspace lattice} $\mathcal{L}$ on a Hilbert space $\Hh$ is a collection of closed (under norm
topology) subspaces of $\Hh$ which is closed under the formation
of arbitrary intersection (denoted by
$\wedge$) and closed linear span (denoted by
$\vee$), and which includes $\{0\}$ and $\Hh$. If $ \mathcal{L} $ is subspace lattice of $\Hh$ and $L\in\mathcal{L}$, we define
\[ L_{-}=\vee \lbrace M\in \mathcal{L} \, \vert \, L\nsubseteq M\rbrace, \]
\[  L_{+}=\wedge \lbrace M\in \mathcal{L} \, \vert \, M\nsubseteq L\rbrace.\]
A totally ordered subspace lattice $ \mathcal{N}$ on $\X$ is called a \emph{nest}. A subspace lattice $ \mathcal{L}$ on a Hilbert space $ \mathbb{H}$
is called a \emph{commutative subspace lattice}, or a \emph{CSL},
if the projections of $ \mathbb{H}$ onto the subspaces of $\mathcal{L}$ commute with each other. A subspace lattice $\mathcal{L}$ is said to be \emph{completely distributive} if $L=\vee\lbrace M\in \mathcal{L} \, \vert \, L\nsubseteq M_{-}\rbrace$ for every $L\in \mathcal{L}$ with $L\neq \lbrace 0 \rbrace$. When $\mathcal{L} \neq \{\{0\},\Hh\}$, we say that $\mathcal{L}$ is \emph{non-trivial}.
\par
For a subspace lattice $ \mathcal{L}$, we define the \emph{associated subspace lattice} $Alg\mathcal{L}$ by
\[ Alg\mathcal{L}=\{ T\in \mathcal{B}(\Hh)\,|\, T(L)\subseteq L
\,for\, all \,L\in \mathcal{L}\}. \]
Obviously, $Alg\mathcal{L}$ is a unital weakly closed subalgebra of $ \mathcal{B}(\Hh) $. Dually, if $\A$ is a subalgebra of $ \mathcal{B}(\Hh) $, by $Lat\A$ we denote the collection of closed subspaces of $\Hh$ that are left invariant by each operator in $\A$. An algebra $\A\subseteq\mathcal{B}(\Hh) $ is \emph{reflexive} if $\A=AlgLat\A$. Clearly, every reflexive algebra is of the form $Alg\mathcal{L}$ for some subspace lattice and vice versa. We call $Alg\mathcal{L}$ a \emph{CSL-algebra} if $\mathcal{L}$ is a commutative subspace lattice, and a \emph{CDC-algebra} if $\mathcal{L}$ is completely distributive $CSL$. Also for a nest $ \mathcal{N} $, the algeba $Alg\mathcal{N}$ is called a \emph{nest algebra}. Recall that a $CSL$-algebra
$Alg\mathcal{L}$ is irreducible if and only if the commutant is trivial, i.e. $(Alg\mathcal{L})^{\prime}=\mathbb{C}I $. In particular, nest algebras are irreducible $CDC$-algebras.
\section{Main results}
First, we characterize continuous linear maps of unital Banach algebras which are like Jordan left derivations at unit-Jordan product elements.
\par
In order to prove our results we need the following result.
\begin{lem} \label{c1}
Let $\A$ be a Banach algebra with unity $\textbf{1}$ and $ \M $ be a unital Banach left $ \A $-module. let $ \delta: \A \rightarrow \M$ be a continuous
linear map with the property that
\[ a,b\in Inv_{0}(\A) \Rightarrow a\delta(a^{-1})+a^{-1}\delta(a)=\delta(\textbf{1}).\]
Then
$\delta$ is a Jordan left derivation.
\end{lem}
\begin{proof}
Since $\textbf{1}\in Inv_{0}(\A) $, it follows that $2\delta(\textbf{1})=\delta(\textbf{1})$. Hence $\delta(\textbf{1})=0$.
\par
Let $a$ be in $\A$. For each scalar $\lambda \in \mathbb{C}$, we
have $e^{\lambda a}\delta(e^{-\lambda a})+e^{-\lambda a}\delta(e^{\lambda a})=0$, since
$e^{\A}\subseteq Inv_{0}(\A)$. Thus
\begin{equation*}
\begin{split}
0&=e^{\lambda a}\delta(e^{-\lambda a})+e^{-\lambda a}\delta(e^{\lambda a})\\&=\Sigma_{n=0}^{\infty}\frac{\lambda^{n}a^{n}}{n!}\delta(\Sigma_{m=0}^{\infty}\frac{(-1)^{m}\lambda^{m}a^{m}}{m!})+\Sigma_{m=0}^{\infty}\frac{(-1)^{m}\lambda^{m}a^{m}}{m!}\delta(\Sigma_{n=0}^{\infty}\frac{\lambda^{n}a^{n}}{n!})\\&=\Sigma_{m=0}^{\infty}\Sigma_{n=0}^{\infty}\frac{(-1)^{m}\lambda^{m+n}}{m!n!}(a^{n}\delta(a^{m})+a^{m}\delta(a^{n}))\\&
=\Sigma_{k=0}^{\infty}\lambda^{k}(\sum_{m+n=k}\frac{(-1)^{m}}{m!n!}(a^{n}\delta(a^{m})+a^{m}\delta(a^{n})),
\end{split}
\end{equation*}
since $\delta$ is a continuous linear map. Consequently,
\begin{equation}\label{e1}
\sum_{m+n=k}\frac{(-1)^{m}}{m!n!}(a^{n}\delta(a^{m})+a^{m}\delta(a^{n}))=0
\end{equation}
for all $a\in \A$ and $k\geq 0$. Taking $k=2$ in \eqref{e1}, we obtain
\[
\frac{1}{2}(\delta(a^{2})+a^{2}\delta(\textbf{1}))-(a\delta(a)+a\delta(a))+\frac{1}{2}
(a^{2}\delta(\textbf{1})+\delta(a^{2}))=0,
\]
for any $a\in \A$. So from $\delta(\textbf{1})=0$, we have
\[\delta(a^{2})=2a\delta(a), \, \, \, (a\in \A).\]
\end{proof}
\begin{prop}\label{ld1}
Let $\A$ be a Banach algebra with unity $1$ and $\M$ be a unital
Banach left $\A$-module. Let $\delta:\A \rightarrow \M$ be a
continuous linear map satisfying
\[a,b\in \A, \quad ab+ba=\textbf{1} \Rightarrow 2a\delta(b)+2b\delta(a)=\delta(\textbf{1}), \quad (*)\]
then $\delta$ is a Jordan left derivation.
\end{prop}
\begin{proof}
Let $a\in Inv_{0}(\A)$ be arbitrary. Since $(\frac{1}{2}a)a^{-1}+a^{-1}(\frac{1}{2}a)=\textbf{1}$, it follows that
\[ 2(\frac{1}{2}a)\delta(a^{-1})+2a^{-1}\delta(\frac{1}{2}a)=\delta(\textbf{1}).\]
So
\[a\delta(a^{-1})+a^{-1}\delta(a)=\delta(\textbf{1}),\]
for all $a\in Inv_{0}(\A)$. Therefore from Lemma~\ref{c1}, $\delta$ is a Jordan left derivation.
\end{proof}
\begin{cor}\label{ld2}
Let $\A$ be a Banach algebra with unity $1$ and $\M$ be a unital
Banach left $\A$-module. Let $x,y\in Z(\A)$ with $x+y=\textbf{1}$ and let
$\delta:\A \rightarrow \M$ be a continuous linear map. If
$\delta$ satisfying
\[a,b\in \A, \quad ab+ba=x \Rightarrow 2a\delta(b)+2b\delta(a)=\delta(x), \]
and
\[a,b\in \A, \quad ab+ba=y \Rightarrow 2a\delta(b)+2b\delta(a)=\delta(y). \]
Then $\delta$ is a Jordan left derivation.
\end{cor}
\begin{proof}
For $a,b\in \A$ with $ab+ba=1$, we have $abx+bax=x$ and $aby+bay=y$. So $axb+bax=x$ and $ayb+bay=y$. From hypothesis, it follows that
\[ \delta(x)=2ax\delta(b)+2b\delta(ax), \]
and
\[ \delta(y)=2ay\delta(b)+2b\delta(ay). \]
Combining the two above equations, we get that
\[
\delta(\textbf{1})=\delta(x+y)=2a\delta(b)+2b\delta(a).
\]
So from Proposition~\ref{ld1}, $\delta$ is a Jordan left derivation.
\end{proof}
If $\A$ is a $CSL$-algebra or a unital semisimple Banach algebra,
then by \cite{Ji} and \cite{Vuk} every continuous Jordan left
derivation on $\A$ is zero. Hence from Proposition~\ref{ld1}, we have the next corollary.
\begin{cor}
Let $\A$ be a $CSL$-algebra or a unital semisimple Banach algebra, and let $\delta:\A\rightarrow \A$ be a continuous linear map satisfying
\[a,b\in \A, \quad ab+ba=\textbf{1} \Rightarrow 2a\delta(b)+2b\delta(a)=\delta(\textbf{1}). \]
Then $\delta$ is zero.
\end{cor}
We continue by characterizing the continuous linear maps behaving like Jordan left derivations at nontrivial idempotent-Jordan product elements on reflexive algebras.
\begin{thm}\label{asli}
Let $Alg \mathcal{L}$ be a reflexive algebra on a Hilbert space $ \Hh $, and there exists a non-trivial idempotent $P\in Alg\mathcal{L}$ with range
$P(\Hh) \in \mathcal{L}$. If $\delta : Alg\mathcal{L}\rightarrow
Alg\mathcal{L}$ is a continuous linear map, then $\delta$ satisfying
\[A,B\in \A, \quad AB+BA=P \Rightarrow 2A\delta(B)+2B\delta(A)=\delta(P), \quad (**) \]
if and only if $\delta(A)=\alpha(A)+\beta(A)+PA(I-P)\delta(I)$, where
\begin{enumerate}
\item[(i)] $\alpha: \A \rightarrow \A$ is a continuous linear map which is a Jordan left derivation and $\alpha(A)=P\alpha(PAP)$ for all $A\in \A$;
\item[(ii)] $\beta:\A\rightarrow \A$ is a continuous linear map satisfying
\[ \beta(A)=(I-P)\beta((I-P)A(I-P))(I-P) \quad for \, \, all \, \, A\in \A, \]
\begin{equation*}
\begin{split}
A,B\in \A, & \quad (I-P)A(I-P)B(I-P)+(I-P)B(I-P)A(I-P)=0\\& \Rightarrow (I-P)A\beta(B)+(I-P)B\beta(A)=0,
\end{split}
\end{equation*}
and
\[ PA(I-P)\beta(B)=PA(I-P)B(I-P)\delta(I) \quad for \, \, all \, \, A,B\in \A. \]
\end{enumerate}
\end{thm}
\begin{proof}
As a notational convenience, we denote $\A=Alg\mathcal{L}$, $\p=P$, $\q=I-P$, $\A_{11}=\p \A \p$, $\A_{12}=\p\A \q$ and $\A_{22}=\q\A \q$. Then we have $\q \A \p=\{0\}$ and hence
\[ \A= \A_{11}\dot{+} \A_{12}\dot{+} \A_{22} \]
as sum of linear spaces. This is so-called the Peirce
decomposition of $\A =Alg\mathcal{L}$. The sets $\A_{11}$,
$\A_{12}$ and $\A_{22}$ are
closed in $\A$. In fact $\A_{11}$ and
$\A_{22}$ are Banach subalgebras of
$\A$ with unity $\p$ and $\q$, respectively and
$\A_{12}$ is a unital Banach
$(\A_{11},\A_{22})$-bimodule. Throughout the proof, $A_{ij}$ and $B_{ij}$ will denote arbitrary elements in $\A_{ij}$ for $1\leq i,j \leq 2$.
\par
Assume that $\delta$ satisfies $(**)$. For $A,B\in \A$ with $AB+BA=2P$, we have $(\frac{1}{2}A)B+B(\frac{1}{2}A)=P$. So it follows that
\[A,B\in \A, \quad AB+BA=2P \Rightarrow A\delta(B)+B\delta(A)=\delta(P). \]
For any $A_{11} \in Inv(\A_{11})$ and $A_{22}\in \A_{22}$, since $A_{11}(A_{11}^{-1}+A_{22})+(A_{11}^{-1}+A_{22})A_{11}=2\p$, ($A_{11}^{-1}$
is the inverse of $A_{11}$ in $\A_{11}$) we have
\begin{equation}\label{e1}
A_{11}\delta(A_{11}^{-1}+A_{22})+(A_{11}^{-1}+A_{22})\delta(A_{11})
=\delta(\p)
\end{equation}
Multiplying this identity by $\q$ both on the left and on the right we find \[A_{22}\delta(A_{11})\q=\q\delta(\p)\q.\]
Now taking $A_{22}=\q$ in this equation, we obtain $\q\delta(A_{11})\q=\q\delta(\p)\q$ and hence $2\q\delta(A_{11})\q=\q\delta(A_{11})\q$ for all $A_{11}\in Inv(\A_{11})$. So $\q\delta(A_{11})\q=0$ for all $A_{11}\in Inv(\A_{11})$. Since any element in a Banach algebras is a sum of invertible
elements, by the linearity of $\delta$ we have
\begin{equation}\label{e2}
\q\delta(A_{11})\q=0,
\end{equation}
for all $A_{11}\in \A_{11}$. Multiplying the Equation~\eqref{e1} by $\p$ both on the left and on the right we arrive at
\[A_{11}\delta(A_{11}^{-1}+A_{22})\p+A_{11}^{-1}\delta(A_{11})\p
=\p\delta(\p)\p.\]
Now letting $A_{11}=\p$ in this equantion, we get $\p\delta(A_{22})\p=-\p\delta(\p)\p$ and therefore $2\p\delta(A_{22})\p=\p\delta(A_{22})\p$ for all $A_{22}\in \A_{22}$. So
 \begin{equation}\label{e3}
\p\delta(A_{22})\p=0,
\end{equation}
for all $A_{22}\in \A_{22}$. Now, multiplying the Equation~\eqref{e1}, from the left by $\p$ and from the right by $\q$, it follows that
\[ A_{11}\delta(A_{11}^{-1}+A_{22})\q+A_{11}^{-1}\delta(A_{11})\q
=\p\delta(\p)\q. \]
Taking $A_{11}=\p$ in this equantion and by a similar arguments as above we have
\begin{equation}\label{e4}
\p\delta(A_{22})\q=0,
\end{equation}
for all $A_{22}\in \A_{22}$.
\par
Since $(A_{11}+A_{12})(A_{11}^{-1}-A_{11}^{-1}A_{12}A_{22}-A_{11}^{-2}A_{12}+A_{22})+(A_{11}^{-1}-A_{11}^{-1}A_{12}A_{22}-A_{11}^{-2}A_{12}+A_{22})(A_{11}+A_{12})=2\p$, for each $A_{11}\in Inv(\A_{11})$, $A_{12}\in \A_{12}$ and $A_{22}\in \A_{22}$, we have
\begin{equation}\label{e5}
\begin{split}
(A_{11}&+A_{12})\delta(A_{11}^{-1}-A_{11}^{-1}A_{12}A_{22}-A_{11}^{-2}A_{12}+A_{22})\\&+(A_{11}^{-1}-A_{11}^{-1}A_{12}A_{22}-A_{11}^{-2}A_{12}+A_{22})\delta(A_{11}+A_{12})=\delta(\p),
\end{split}
\end{equation}
for all $A_{11}\in Inv(\A_{11})$, $A_{12}\in \A_{12}$ and $A_{22}\in \A_{22}$. Multiplying the Equation~\eqref{e1} by $\p$ both on the left and on the right and by the fact that $\q \A \p=\lbrace 0 \rbrace$, we arrive at
\[A_{11}\delta(A_{11}^{-1}-A_{11}^{-1}A_{12}A_{22}-A_{11}^{-2}A_{12}+A_{22})\p + A_{11}^{-1}\delta(A_{11}+A_{12})\p =\p\delta(\p )\p.\]
Now letting $A_{11}=\p$ and $A_{22}=\q$ in this identity and by the Equation~\eqref{e3}, we see that $\p\delta(A_{12})\p=\p \delta(\p)\p$ for all $A_{12}\in \A_{12}$. Hence
\begin{equation}\label{e6}
\p\delta(A_{12})\p=0,
\end{equation}
for all $A_{12}\in \A_{12}$. Multiplying the Equation~\eqref{e5} by $\q$ both on the left and on the right and by the Equation~\eqref{e2}, we get $A_{22}\delta(A_{12})\q=\q \delta(\p)\q$. Replacing $A_{22}$ by $\q$, we find $\q\delta(A_{12})\q=\q \delta(\p)\q$ for all $A_{12}\in \A_{12}$. So
\begin{equation}\label{e7}
\q\delta(A_{12})\q=0,
\end{equation}
for all $A_{12}\in \A_{12}$. Now, multiplying the Equation~\eqref{e5}, from the left by $\p$, from the right by $\q$ and by Equations ~\eqref{e2},~\eqref{e4} and ~\eqref{e7}, we see that
\[-A_{11}\delta(A_{11}^{-1}A_{12}A_{22})\q-A_{11}\delta(A_{11}^{-2}A_{12})\q+A_{12}\delta(A_{22})\q
+A_{11}^{-1}\delta(A_{12})\q=\p \delta(\p)\q.\]
Letting $A_{11}=\p$ in this equation, it follows that $-\p \delta(A_{12}A_{22})\q+A_{12}\delta(A_{22})\q=\p \delta(\p)\q$ for all $A_{12}\in \A_{12}$ and $A_{22}\in \A_{22}$. So $2(-\p\delta(A_{12}A_{22})\q+A_{12}\delta(A_{22})\q)
=-\p\delta(A_{12}A_{22})\q+A_{12}\delta(A_{22})\q$ and hence
\begin{equation}\label{e8}
\p\delta(A_{12}A_{22})\q=A_{12}\delta(A_{22})\q,
\end{equation}
for all $A_{12}\in \A_{12}$ and $A_{22}\in \A_{22}$. By taking $A_{22}=\q$ in ~\eqref{e8} and by Equations ~\eqref{e2}, we have
\begin{equation}\label{e9}
\p\delta(A_{12})\q=A_{12}\delta(\q)\q=A_{12}\delta(I),
\end{equation}
for all $A_{12}\in \A_{12}$. Now from Equations ~\eqref{e8} and ~\eqref{e9}, it follows that
\begin{equation}\label{e10}
A_{12}\delta(A_{22})\q=\p\delta(A_{12}A_{22})\q=A_{12}A_{22}\delta(I),
\end{equation}
for all $A_{12}\in \A_{12}$ and $A_{22}\in \A_{22}$.
\par
Define $\alpha:\A \rightarrow \A$ by $\alpha(A)=\p\delta(\p A\p)$. So $\alpha$ is continuous, $\alpha(I)=\alpha(\p)$ and $\alpha(A)=\p\alpha(\p A\p)$ for all $A\in \A$. Consider $A,B\in \A$ with $AB+BA=\p$. Since $\q \A \p =\lbrace 0 \rbrace$, it follows that $\p A \p B \p+\p B \p A\p=\p$ and hence
\[ 2\p A \p \delta(\p B \p)+2\p B \p\delta(\p A \p)=\delta(\p). \]
So
\[ 2A \p \delta(\p B \p)+2B \p\delta(\p A \p)=\p \delta(\p).\]
Therefore
\[ 2A\alpha(B)+2B\alpha(A)=\alpha(I).\]
Thus $\alpha$ satisfies $(*)$ and by Proposition~\ref{ld1}, $\alpha$ is a Jordan left derivation. Now define $\beta:\A \rightarrow \A$ by $\beta(A)=\q \delta(\q A \q)\q$. It is clear that $\beta(A)=\q \beta(\q A \q)\q$ for all $A\in \A$. Let $A,B\in \A$ with $\q A \q B \q +\q B \q A \q=0$. So $(\p+\q A \q)(\p+\q B \q)+(\p+\q B \q)(\p+\q A \q)=2\p$ and hence
\[ (\p+\q A \q)\delta(\p+\q B \q)+(\p+\q B \q)\delta(\p+\q A \q)=\delta(\p).\]
Multiplying this identity by $\q$ both on the left and on the right, by Equation~\eqref{e2} we find
\[ \q A \q\delta(\q B \q)\q+\q B \q\delta(\q A \q)\q=0.\]
So
\[ \q A\beta(B)+\q B \beta(A)=0.\]
Also from Equation~\eqref{e10} we have
\[ \p A \beta(B)=\p A \q B \q \delta(I),\]
for all $A,B\in \A$.
\par
Now by Equations~\eqref{e2}--~\eqref{e4}, ~\eqref{e6}, ~\eqref{e7} and ~\eqref{e9}, it follows that
\[ \delta(A)=\alpha(A)+\beta(A)+\p A \q\delta(I),\]
for all $A\in \A$. Thus $\delta$ has the desired form.
\par
Conversely, assume that  $\delta$ satisfies the given conditions. For any $A,B\in \A$ with $AB+BA=P$ by the fact that $\q \A \p=\lbrace 0 \rbrace$, we see that
\begin{equation}\label{e11}
\begin{split}
&\p A \p B \p +\p B \p A \p =\p; \\ &
\p A \p B \q +\p A \q B \q + \p B \p A \q + \p B \q A \q =0; \\ &
\q A \q B \q +\q B \q A \q =0.
\end{split}
\end{equation}
Now from Equations~\eqref{e11} and assumptions, for any $A,B\in \A$ with $AB+BA=P$, we have
\begin{equation*}
\begin{split}
 2A\delta(B)+2B\delta(A)&=2A(\alpha(B)+\beta(B)+\p B \q\delta(I))+2B(\alpha(A)+\beta(A)+\p A \q \delta(I))\\& = 2\p A \p \alpha(\p B \p)+ 2\p B \p \alpha(\p A \p) + 2\q A \q \beta(\q B \q)\q \\ &+2\q B \q \beta(\q A \q)\q + 2\p A \q \beta(\q B \q)\q +2\p B \q \beta(\q A \q)\q \\ & + 2\p A \p B \q\delta(I) + 2\p B \p A \q\delta(I)\\& = \alpha(\p) +2\p A \q B \q \delta(I)+2\p B \q A \q \delta(I)\\ &  + 2\p A \p B \q\delta(I) + 2\p B \p A \q\delta(I)\\&=\alpha(\p)=\delta(\p).
\end{split}
\end{equation*}
So $\delta$ satisfies $(**)$ .
\end{proof}
Note that if $P$ is a non-trivial idempotent in
$Alg \mathcal{L}$ satisfying $PP_{L}=P_{L}$ and $P_{L}P=P$ for some non-trivial element $L\in \mathcal{L}$, then $P(\Hh)=L\in \mathcal{L}$.
\par
If $Alg \mathcal{L}$ is a non-trivial $CSL$-algebra on a Hilbert space $ \Hh $, then for every non-trivial element $L\in \mathcal{L}$, we have $P_{L}\in Alg \mathcal{L}$. Also every continuous Jordan left derivation on a $CSL$ algebra
is zero \cite{Ji}. From these facts and Theorem~\ref{asli} we have the following corollary.
\begin{cor}\label{csl}
Let $Alg \mathcal{L}$ be a non-trivial $CSL$-algebra on a Hilbert space $ \Hh $ and $L\in \mathcal{L}$ be non-trivial. If $\delta : Alg\mathcal{L}\rightarrow
Alg\mathcal{L}$ is a continuous linear map, then $\delta$ satisfying
\[A,B\in Alg\mathcal{L}, \quad AB+BA=P_{L} \Rightarrow 2A\delta(B)+2B\delta(A)=\delta(P_{L}), \]
if and only if $\delta(A)=\beta(A)+P_{L}A(I-P_{L})\delta(I)$ for all $A \in Alg \mathcal{L}$, where
$\beta:Alg\mathcal{L}\rightarrow Alg\mathcal{L}$ is a continuous linear map satisfying
\[ \beta(A)=(I-P_{L})\beta((I-P_{L})A(I-P_{L}))(I-P_{L}) \quad for \, \, all \, \, A\in Alg\mathcal{L}, \]
\begin{equation*}
\begin{split}
A,B\in Alg\mathcal{L}, & \quad (I-P_{L})A(I-P_{L})B(I-P_{L})+(I-P_{L})B(I-P_{L})A(I-P_{L})=0\\& \Rightarrow (I-P_{L})A\beta(B)+(I-P_{L})B\beta(A)=0,
\end{split}
\end{equation*}
and
\[ P_{L}A(I-P_{L})\beta(B)=P_{L}A(I-P_{L})B(I-P_{L})\delta(I) \quad for \, \, all \, \, A,B\in Alg\mathcal{L}. \]
\end{cor}
To prove the next corollary, we need the following lemma from \cite [Theorem 3.4]{Lu}.
\begin{lem}\label{ann}
Let $Alg \mathcal{L}$ be an irreducible $CDC$ algebra on a Hilbert space $\Hh$, then there is a non-trivial element $L\in \mathcal{L}$ such that for $A\in Alg \mathcal{L}$, $P_{L}Alg \mathcal{L}(I-P_{L})A=\lbrace 0 \rbrace$ implies $(I-P_{L})A=0$.
\end{lem}
\begin{cor}\label{cdc}
Let $Alg \mathcal{L}$ be an irreducible $CDC$ algebra on a Hilbert space $\Hh$, and let $P_{L} \in Alg \mathcal{L}$ be the projection in Lemma~\ref{ann}. If $\delta : Alg\mathcal{L}\rightarrow
Alg\mathcal{L}$ is a continuous linear map, then $\delta$ satisfying
\[A,B\in Alg\mathcal{L}, \quad AB+BA=P_{L} \Rightarrow 2A\delta(B)+2B\delta(A)=\delta(P_{L}), \]
if and only if $\delta(A)=A\delta(I)$ for all $A \in Alg \mathcal{L}$ and $P_{L}\delta(I)=0$.
\end{cor}
\begin{proof}
First, we give the proof of 'only if' part. By Corollary~\ref{csl}, we have $\delta(A)=\beta(A)+P_{L}A(I-P_{L})\delta(I)$ for all $A \in Alg \mathcal{L}$, where
$\beta:Alg\mathcal{L}\rightarrow Alg\mathcal{L}$ is a continuous linear map satisfying
\[ \beta(A)=(I-P_{L})\beta((I-P_{L})A(I-P_{L}))(I-P_{L}) \quad for \, \, all \, \, A\in \A, \]
and
\[ P_{L}A(I-P_{L})\beta(B)=P_{L}A(I-P_{L})B(I-P_{L})\delta(I) \quad for \, \, all \, \, A,B\in \A. \]
So $\delta(P_{L})=\beta(P_{L})=0$ and $P_{L}\delta(I)=P_{L}\delta(I-P_{L})=P_{L}\beta(I-P_{L})=0$. Also we have $P_{L}Alg \mathcal{L}(I-P_{L})(\beta(B)-B(I-P_{L})\delta(I))=\lbrace 0 \rbrace$ for all $B \in Alg \mathcal{L}$. Therefore by Lemma~\ref{ann}, we have
\[ \beta(B)=(I-P_{L})B(I-P_{L})\delta(I), \]
for all $B \in Alg \mathcal{L}$. Now from these results, it follows that
\begin{equation*}
\begin{split}
 \delta(A)&=\beta(A)+P_{L}A(I-P_{L})\delta(I)\\&=(I-P_{L})A(I-P_{L})\delta(I)+P_{L}A(I-P_{L})\delta(I)\\&=A(I-P_{L})\delta(I)\\&=A\delta(I),
\end{split}
\end{equation*}
for all $A \in Alg \mathcal{L}$.
\par
Next, we check the 'if part'. For any $A,B\in Alg \mathcal{L}$ with $AB+BA=P_{L}$, we have
\begin{equation*}
\begin{split}
2A\delta(B)+2B\delta(A)&=2AB\delta(I)+2BA\delta(I)\\&
=2P_{L}\delta(I)\\&=0=\delta(P_{L}),
\end{split}
\end{equation*}
since $\delta(P_{L})=P_{L}\delta(I)=0$. So $\delta$ has the desired form.
\end{proof}
Let $Alg \mathcal{N}$ be a non-trivial nest algebra on a Hilbert space $ \Hh $, then for every non-trivial element $N\in \mathcal{N}$, $P_{N}Alg \mathcal{N}(I-P_{N})A=\lbrace 0 \rbrace$ implies $(I-P_{N})A=0$ ($A\in Alg \mathcal{N} $). By using similar arguments as that in the proof of Corollary~\ref{cdc}, we get the next corollary.
\begin{cor}\label{nest}
Let $Alg \mathcal{N}$ be a nest algebra on a Hilbert space $\Hh$, and let $N \in \mathcal{N}$ be a non-trivial element. If $\delta : Alg\mathcal{N}\rightarrow
Alg\mathcal{N}$ is a continuous linear map, then $\delta$ satisfying
\[A,B\in Alg\mathcal{N}, \quad AB+BA=P_{N} \Rightarrow 2A\delta(B)+2B\delta(A)=\delta(P_{N}), \]
if and only if $\delta(A)=A\delta(I)$ for all $A \in Alg \mathcal{N}$ and $P_{N}\delta(I)=0$.
\end{cor}



\bibliographystyle{amsplain}
\bibliography{xbib}

\end{document}